\documentclass[12 pt]{article}%
\usepackage{amsmath, amsfonts, amsthm, color,latexsym}
\usepackage{amssymb}
\usepackage{color}
\usepackage{graphicx}
\usepackage{amsmath}
\usepackage{amsfonts}%
\providecommand{\U}[1]{\protect\rule{.1in}{.1in}}
\newtheorem{theorem}{Theorem}[section]
\newtheorem{proposition}[theorem]{Proposition}
\newtheorem{corollary}[theorem]{Corollary}

\newtheorem{lemma}[theorem]{Lemma}
\newtheorem{final remark}[theorem]{Final Remark}
\newtheorem{definition}[theorem]{Definition}
\begin{document}

\title{\textsc{Coincidence results for summing multilinear mappings}}
\date{}
\author{Oscar Blasco\thanks{Supported by Ministerio de Ciencia e Innovaci\'{o}n
MTM2011-23164}\,, Geraldo Botelho\thanks{Supported by CNPq Grant
302177/2011-6 and Fapemig Grant PPM-00326-13.}\,, Daniel Pellegrino\thanks{Supported by CNPq
Grant 477124/2012-7.}~ and Pilar Rueda\thanks{Supported by Ministerio de
Ciencia e Innovaci\'{o}n MTM2011-22417 .\hfill\newline2000 Mathematics Subject
Classification. Primary 46G25; Secondary 47B10. \newline Keywords: absolutely
summing, projective tensor product, Cohen summing, multilinear mapping.}}
\maketitle

\begin{abstract}
In this paper we prove coincidence results concerning spaces of absolutely
summing multilinear mappings between Banach spaces. The nature of these
results arises from two distinct approaches: the coincidence of two \textit{a
priori} different classes of summing multilinear mappings and the summability
of all multilinear mappings defined on products of Banach spaces. Optimal
generalizations of known results are obtained. We also introduce and explore
new techniques in the field, for example a technique to extend coincidence
results for linear, bilinear and even trilinear mappings to general
multilinear ones.

\end{abstract}

\section*{Introduction}

The theory of absolutely summing multilinear mappings between Banach spaces
has its root in the research program designed by A. Pietsch in \cite{Pie}, as
an attempt to generalize the linear operator theory to a multilinear context.
A great amount of research in this direction has been done since then and
applications of the theory of nonlinear summing mappings to other fields have
been found, for example it has been used in the study of the maximal domain of
convergence of vector-valued Dirichlet series \cite{DeGaMaPe} and in Quantum
Information Theory \cite{monta, PWPVJ}.

By definition, absolutely summing multilinear mappings improve the summability
of sequences in Banach spaces and this is why many researchers have focussed
much of their interest on the study of these mappings. The general purpose of
these studies is to obtain and to improve summability conditions for
multilinear mappings. A desirable result in the theory is what has been called
in the literature a \textit{coincidence result}. This consists in finding
examples of Banach spaces $E_{1},\ldots,E_{n},F$, or general conditions on
them, that ensure a good summability behavior of every continuous $n$-linear
mapping from $E_{1}\times\cdots\times E_{n}$ to $F$ or improve the summability
behavior of those mappings that already enjoy some summing property. A
cornerstone in the linear theory is Grothendieck's theorem, that asserts that
every continuous linear operator from $\ell_{1}$ to $\ell_{2}$ is absolutely
summing. Grothendieck's theorem has been a permanent source of inspiration in
the search of linear and multilinear coincidence results. The Defant-Voigt
Theorem (see \cite[Corollary 3.2]{Port}, \cite[Proposition 17.5.1]{Pietsch} or
\cite{BBPR}, where an improved version can be found), which states that any
continuous multilinear functional is absolutely $(1;1,\ldots,1)$-summing, is
probably the first coincidence result for multilinear mappings. We mention a
few more examples. The fact that bilinear forms on either an $\mathcal{L}%
_{\infty}$-space or the disc algebra or the Hardy space are $2$-dominated was
proved in \cite[Theorem 3.3]{bote97} and \cite[Proposition 2.1]{pams},
respectively. In \cite[Theorem~3.7]{BBPR} the authors use this bilinear
coincidence to show that all $n$-linear forms defined on a product
$E_{1}\times\cdots\times E_{n}$ of Banach spaces is $(1;2,\ldots,2)$-summing
whenever $E_{1}=E_{2}$ and each $E_{j}$ is either an $\mathcal{L}_{\infty}%
$-space, the disc algebra $\mathcal{A}$ or the Hardy space $\mathcal{H}%
^{\infty}$. It is worth mentioning that this situation of lifting properties
from bilinear to multilinear mappings is non-trivial in many cases. Indeed, it
is not true that the multilinear theory follows by induction from the linear
case. Many examples of the difficulties of lifting the linear theory to the
multilinear setting can be found in the literature (see, e.g., \cite{Junek}).

This paper is concerned with coincidence results in the theory of absolutely
summing multilinear mappings. Section \ref{Not} is devoted to fix the notation
and to recall some definitions and basic facts. In Section \ref{D} we
investigate two notions related to that of absolutely summing multilinear
mappings, namely, \textit{weakly $(p;p_{1},p_{2},\ldots,p_{n})$-summing}
multilinear mappings (cf. Definition \ref{def1}) and
\textit{\ Cohen $(p;p_{1},p_{2},\ldots,p_{n})$-nuclear} multilinear mappings
(cf. Definition \ref{def2}).
Relations between these notions and between them and the usual concept of
absolutely summing multilinear mappings are established. In Section \ref{lift}
we use the cotype of the Banach spaces $E_{1},\ldots,E_{n}$ and find some
conditions on the positive numbers $p,p_{1},\ldots,p_{n},q,q_{1},\ldots,q_{n}$
to ensure that $(p;p_{1},\ldots,p_{n})$ and $(q;q_{1},\ldots,q_{n})$ summing
mappings coincide. We apply the results of this section to get an optimal
generalization of results of \cite{michels, Junek, pop}. The results presented
in this section also generalize some results of \cite{thi, popaarchiv}. In
Section \ref{section3} we get conditions that ensure that all continuous
multilinear mappings on Banach spaces $E_{1},\ldots,E_{n}$ are $(p;p_{1}%
,\ldots,p_{n})$-summing. We show how to lift summability properties of
bilinear mappings defined on $E_{2i-1}\times E_{2i}$ to $n$-linear mappings
defined on $E_{1}\times\cdots\times E_{n}$. We prove that if any bilinear form
defined on $E^{2}$ is $(1;r,r)$-summing and any trilinear mapping on $E^{3}$
is $(1;r,r,r)$-summing, $1\leq r\leq2$, then any $n$-linear mapping on $E^{n}$
is $(1;r,\ldots,r)$-summing. We also characterize when all multilinear
mappings are $(p;p_{1},\ldots,p_{n})$-summing by means of projective tensor
products of vector valued sequences spaces. A close connection between Cohen
summability and the Littlewood-Orlicz property is also provided.

\section{Notation and background}

\label{Not}

All Banach spaces are considered over the scalar field $\mathbb{K}=\mathbb{R}
$ or $\mathbb{C}$. Given a Banach space $E$, let $B_{E}$ denote the closed
unit ball of $E$ and $E^{\prime}$ denote its topological dual.

Let $p>0$. By $\ell_{p}(E)$ we denote the ($p$-)Banach space of all absolutely
$p$-summable sequences $(x_{j})_{j=1}^{\infty}$ in $E$ endowed with its usual
$\ell_{p}$-norm ($p$-norm if $0<p<1$).
Let $\ell_{p}^{w}(E)$ be the space of those sequences $(x_{j})_{j=1}^{\infty}$
in $E$ such that $(\varphi(x_{j}))_{j=1}^{\infty}\in\ell_{p}$ for every
$\varphi\in E^{\prime}$ endowed with the norm ($p$-norm if $0<p<1$)
\[
\Vert(x_{j})_{j=1}^{\infty}\Vert_{\ell_{p}^{w}(E)}=\sup_{\varphi\in
B_{E^{\prime}}}\left(  \sum_{j=1}^{\infty}|\varphi(x_{j})|^{p}\right)
^{\frac{1}{p}}.
\]

For $1\leq p\leq\infty$ let $p^{\prime}$ be the conjugate of $p$, i.e.,
$\frac{1}{p}+\frac{1}{p^{\prime}}=1$. We denote by $\ell_{p}\langle E\rangle$
the Banach spaces of sequences $(x_{j})_{j=1}^{\infty}$ in $E$ such that
\[
\Vert(x_{j})_{j=1}^{\infty}\Vert_{\ell_{p}\langle E\rangle}=\sup\left\{
\sum_{j=1}^{\infty}|\langle x_{j},y_{j}^{\ast}\rangle|:\Vert(y_{j}^{\ast
})_{j=1}^{\infty}\Vert_{\ell_{p^{\prime}}^{w}(E^{\prime})}=1\right\}
<\infty.
\]
Obviously one has for $p\geq1$ that
\[
\ell_{p}\hat{\otimes}_{\pi}E\subseteq\ell_{p}\langle E\rangle\subseteq\ell
_{p}(E)\subseteq\ell_{p}^{w}(E).
\]
The space $\ell_{p}\langle E\rangle$ was first introduced in \cite{C} and it
has been recently described in different ways (see \cite{AB1} for a
description as the space of integral operators from $\ell_{p^{\prime}}$ into
$X$ or \cite{BD} and \cite{FR} for the identification with the projective
tensor product $\ell_{p}\langle E\rangle=\ell_{p}\hat{\otimes}_{\pi}E$). In
the particular case of dual spaces, using the weak principle of local
reflexivity (see \cite[page 73]{DF}) one has that a sequence $(x_{j}^{\ast
})_{j=1}^{\infty}$ in $E^{\prime}$ belongs to $\ell_{p}\langle E^{\prime
}\rangle$ if
\begin{equation}
\Vert(x_{j}^{\ast})_{j=1}^{\infty}\Vert_{\ell_{p}\langle E^{\prime}\rangle
}=\sup\left\{  \sum_{j=1}^{\infty}|\langle x_{j}^{\ast},y_{j}\rangle
|:\Vert\left(  y_{j}\right)  _{j=1}^{\infty}\Vert_{\ell_{p^{\prime}}^{w}%
(E)}=1\right\}  <\infty. \label{dualform}%
\end{equation}

Recall that, for $1\leq q\leq p<\infty$, an operator $T\in\mathcal{L}(E;F)$ is
absolutely \textit{$(p,q)$-summing} if there is $C>0$ such that
\begin{equation}
\left(  \sum_{j=1}^{m}\left\Vert T\left(  x_{j}\right)  \right\Vert
^{p}\right)  ^{\frac{1}{p}}\leq C\Vert(x_{j})_{j=1}^{m}\Vert_{\ell_{q}^{w}(E)}
\label{pqas}%
\end{equation}
for any finite sequences $x_{1},\ldots,x_{m}\in E$. $\Pi_{(p;q)}(E;F)$ denotes
the space of $(p,q)$-summing operators with the norm ($p$-norm if $0<p<1$)
given by the infimum of the constants satisfying (\ref{pqas}).

For $1\leq p,q<\infty,$ an operator $T\in\mathcal{L}(E;F)$ is \textit{Cohen
$(p,q)$-nuclear} (see \cite[page 56]{apiola}) if there is $C>0$ such that
\begin{equation}
\sum_{j=1}^{m}|\langle T(x_{j}),y_{j}^{\ast}\rangle|\leq C\Vert(x_{j}%
)_{j=1}^{m}\Vert_{\ell_{q}^{w}(E)}\cdot\Vert(y_{j}^{\ast})_{j=1}^{m}%
\Vert_{\ell_{p^{\prime}}^{w}(F^{\prime})} \label{cohen}%
\end{equation}
for any finite sequences $x_{1},\ldots,x_{m}\in E$ and $y_{1}^{\ast}%
,\ldots,y_{m}^{\ast}\in F^{\prime}$. $CN_{(p;q)}(E;F)$ denotes the space of
Cohen $(p,q)$-nuclear operators with the norm given by the infimum of the
constants satisfying (\ref{cohen}). This notion was introduced by Cohen
\cite{C} for $p=q$.

Note that $T\in CN_{(p;q)}(E;F)$ is equivalent to the correspondence
\[
(x_{j})_{j}\in\ell_{q}^{w}(E)\mapsto(T(x_{j}))_{j}\in\ell_{p}\langle F\rangle
\]
be well defined (hence linear) and bounded. Hence $CN_{(p;q)}(E;F)\subseteq
\Pi_{(p;q)}(E;F)$ and, due to (\ref{dualform}), $T\in{CN}_{(p;q)}(E;F^{\prime
})$ if there is $C>0$ such that
\begin{equation}
\sum_{j=1}^{m}|\langle T(x_{j}),y_{j}\rangle|\leq C\Vert(x_{j})_{j=1}^{m}%
\Vert_{\ell_{q}^{w}(E)}\Vert(y_{j})_{j=1}^{m}\Vert_{\ell_{p^{\prime}}^{w}(F)}.
\label{cohendual}%
\end{equation}

Now we turn our attention to multilinear maps. Let $E_{1},\ldots,E_{n},E,F$ be
Banach spaces. The Banach space of all continuous $n$-linear mappings from
$E_{1}\times\cdots\times E_{n}$ to $F$ is denoted by $\mathcal{L}(E_{1}%
,\ldots,E_{n};F)$ and endowed with the usual sup norm. We simply write
$\mathcal{L}(^{n}E;F)$ when $E_{1}=\cdots=E_{n}=E$.

For $0<p,p_{1},p_{2},\ldots,p_{n}\leq\infty$ , we assume that $\frac{1}{p}%
\leq\frac{1}{p_{1}}+\cdots+\frac{1}{p_{n}}$. An $n$-linear mapping
$A\in{\mathcal{L}}(E_{1},\ldots,E_{n};F)$ is \textit{absolutely $(p;p_{1}%
,p_{2},\ldots,p_{n})$-summing} if there is $C>0$ such that
\begin{equation}
\Vert(A(x_{j}^{1},x_{j}^{2},\ldots,x_{j}^{n}))_{j=1}^{k}\Vert_{p}\leq
C\prod_{i=1}^{n}\Vert(x_{j}^{i})_{j=1}^{k}\Vert_{\ell_{p_{i}}^{w}(E_{i})}
\label{norm}%
\end{equation}
for all finite families of vectors $x_{1}^{i},\ldots,x_{k}^{i}\in E_{i}$,
$i=1,2,\ldots,n$. The infimum of such $C>0$ is called the $(p;p_{1}%
,\ldots,p_{n})$-summing norm ($p$-norm if $0<p<1$) of $A$ and is denoted by
$\pi_{(p;p_{1},\ldots,p_{n})}(A)$. Let $\Pi_{(p;p_{1},p_{2},\ldots,p_{n}%
)}(E_{1},\ldots,E_{n};F)$ denote the space of all absolutely $(p;p_{1}%
,p_{2},\ldots,p_{n})$-summing $n$-linear mappings from $E_{1}\times
\cdots\times E_{n}$ to $F$ endowed with the norm ($p$-norm if $0<p<1$) $\pi_{(p;p_{1}\ldots,p_{n})}$.

For $n\ge1$ and $A\in{\mathcal{L}}(E_{1},\ldots,E_{n};F)$,
\[
\widehat{A}\colon\ell_{\infty}(E_{1})\times\cdots\times\ell_{\infty}%
(E_{n})\longrightarrow\ell_{\infty}(F)~,~ \widehat{A}((x_{j}^{1}%
)_{j=1}^{\infty},\ldots,(x_{j}^{n})_{j=1}^{\infty}):=(A(x_{j}^{1},\ldots
,x_{j}^{n}))_{j=1}^{\infty},
\]
is a bounded $n$-linear mapping. Given subspaces $X_{i}\subseteq\ell_{\infty
}(E_{i})$ for $1\le i\le n$ and $Y\subseteq\ell_{\infty}(F)$, each one endowed
with its own norm, we say that $\widehat{A}\colon X_{1}\times\cdots\times
X_{n}\longrightarrow Y$ is bounded -- or, equivalently, $\widehat{A}
\in\mathcal{L}(X_{1}, \ldots, X_{n}; Y)$ -- if the restriction of $\widehat A$
to $X_{1} \times\cdots\times X_{n}$ is a well defined (hence $n$-linear)
continuous $Y$-valued mapping.
Clearly $A\in\Pi_{(p;p_{1} ,p_{2},\ldots,p_{n})}(E_{1},\ldots,E_{n};F)$ if and
only if maps $\hat A$ maps $\ell^{w}_{p_{1}}(E_{1})\times\cdots\times\ell
^{w}_{p_{n}}(E_{n})$ boundedly into $\ell_{p}(F)$.

Note also that if $T\in\mathcal{L}(E,F)$ and we write $A_{T}\colon E\times
F^{\prime}\longrightarrow\mathbb{K}$ for the bilinear map
\[
A_{T}(x,y^{\ast})=\langle T(x),y^{\ast}\rangle,x\in E,y^{\ast}\in E^{\prime},
\]
then
\[
T\in CN_{(p;q)}(E;F)\Longleftrightarrow A_{T}\in\Pi_{(1;{q},{p^{\prime}}%
)}(E,F^{\prime};\mathbb{K}).
\]

Absolutely summing mappings fulfill the following inclusion result, which
appears in \cite[Proposition 3.3]{Thesis} (see also \cite{BBPR}), and will be
used several times in this paper:

\begin{theorem}
\label{inclusiontheorem}{\rm(Inclusion Theorem)} \ Let $0<q\leq
p\leq\infty$, $0<q_{j}\leq p_{j}\leq\infty$ for all $j=1,\ldots,n$. If
$\frac{1}{q_{1}}+\cdots+\frac{1}{q_{n}}-\frac{1}{q}\leq\frac{1}{p_{1}}%
+\cdots+\frac{1}{p_{n}}-\frac{1}{p}$, then
\[
\Pi_{(q;q_{1},\ldots,q_{n})}(E_{1},\ldots,E_{n};F)\subseteq\Pi_{(p;p_{1}%
,\ldots,p_{n})}(E_{1},\ldots,E_{n};F)
\]
and $\pi_{(p;p_{1},\ldots,p_{n})}\leq\pi_{(q;q_{1},\ldots,q_{n})}$ for all
Banach spaces $E_{1},\ldots,E_{n},F$.
\end{theorem}

\medskip If $\frac{1}{p}=\frac{1}{p_{1}}+\cdots+\frac{1}{p_{n}}$, absolutely
$(p;p_{1},\ldots,p_{n})$-summing $n$-linear mappings are usually called
\textit{$(p_{1},\ldots,p_{n})$-dominated}. They also satisfy a factorization
result (see \cite[Theorem 13]{Pie}).

For the basic theory of type and cotype in Banach spaces we refer to
\cite[Chapter 11]{Diestel}.


A Banach space $E$ is said to have the \textsl{Orlicz property $q$\/}, $2\le
q<\infty$, if
the identity operator $Id \colon E\longrightarrow E$ is absolutely
$(q,1)$-summing, that is, if $\ell_{1}^{w}(E) \subseteq\ell_{q}(E)$. Clearly
cotype $q$ implies Orlicz property $q$. Some deep results by M. Talagrand
\cite{T1,T2} show that for $q>2$ both notions actually coincide while this is
not the case for $q=2$.

Some of the basic questions in the theory of absolutely summing operators are
to analyze when $\Pi_{(p_{1};q_{1})}(X,Y)=\Pi_{(p_{2};q_{2})}(X,Y)$ or when
$\mathcal{L}(X,Y)=\Pi_{(p;q)}(X,Y)$. Typical results in this direction are
that $\Pi_{(1;1)}(X,Y)=\Pi_{(2;2)}(X,Y)$ for any space $Y$ of cotype 2 or the
fundamental Grothendieck's theorem:
\begin{equation}
{\mathcal{L}}(\ell_{1};\ell_{2})=\Pi_{(1;1)}(\ell_{1};\ell_{2})\quad
\hbox{ or, equivalently, }\quad{\mathcal{L}}(c_{0};\ell_{1})=\Pi_{(2;2)}%
(c_{0};\ell_{1}). \label{gt}%
\end{equation}
Besides the notion of cotype, the property which plays an important role in
our development is the so called \textit{Littlewood-Orlicz property}. A Banach
space $E$ has the \textit{Littlewood-Orlicz property} if $\ell_{1}%
^{w}(E)\subseteq\ell_{2}\langle E\rangle$, that is, $Id\colon E\longrightarrow
E$ is Cohen $(2,1)$-nuclear or the duality map $A_{I}\colon E\times E^{\prime
}\longrightarrow\mathbb{K}$ is $(1;1,2)$-summing. The reader is referred to
\cite[Section 4]{Bu} for a related concept of Littlewood-Orlicz operator. This
is considered in Section 5.

\section{Weakly summing and Cohen nuclear multilinear mappings}

\label{D}

The bilinear form
\[
A\colon\ell_{2}\times\ell_{2}\longrightarrow\mathbb{K}~,~A((\alpha_{j}%
)_{j},(\beta_{j})_{j})=\sum_{j}\alpha_{j}\beta_{j},
\]
is bounded, but $A\notin\Pi_{(1;1,2)}(E_{1},E_{2};\mathbb{K})$ because
$(e_{j})_{j}\in\ell_{2}^{w}(\ell_{2})$, $(\frac{e_{j}}{j})_{j}\in\ell_{1}%
^{w}(\ell_{2})$ and $(A(e_{j},\frac{e_{j}}{j}))_{j}=(\frac{1}{j}e_{j}%
)_{j}\notin\ell_{1}$. This example suggests that no generalization of the
Defant-Voigt theorem can be expected for general $p,p_{1},\ldots,p_{n}$ other
than $p=p_{1}=\cdots=p_{n}=1$. However, let us see that whenever $\frac
{1}{p_{1}}+\frac{1}{p_{2}}+\cdots+\frac{1}{p_{n}}-n+1\geq\frac{1}{p}$ then we
can get a coincidence result for $(p;p_{1},\ldots,p_{n})$-summing $n$-linear
functionals. Moreover, in order to extend such a result to the vector valued
case, we need to consider weakly $(p;p_{1},\ldots,p_{n})$-summing operators.
The essence of the concept of weakly summing multilinear operators, as far as
we know, has its roots in \cite{soares}:


\begin{definition}\rm
\label{def1}For $n\in\mathbb{N}$, let $0<p,p_1,\ldots,p%
_n<\infty$ be such that  $\frac{1}{p_1}+\frac{1}{p_2%
}+\cdots+\frac{1}{p_n}\geq\frac{1}{p}.$
 We say that $A\in\mathcal{L}%
(E_1,\ldots,E_n;F)$ is {\it weakly $(p;p_1,\ldots,p_n)$-summing} if
the induced mapping $\widehat{A}\colon\ell_{p_1}^{w}(E_1)\times\ell
_{p_2}^{w}(E_2)\times\cdots\times\ell_{p_n}^{w}(E_n)\longrightarrow
\ell_{p}^{w}(F)$ is well-defined (hence $n$-linear and bounded). The space
formed by these mappings is denoted by $\Pi_{w(p;p_1,\ldots,p_n)}
(E_1,\ldots,E_n;F)$ and the weakly $(p;p_1,\ldots,p_n)$-summing norm
(p-norm if $0<p<1$) $\pi_{w(p;p_1,\ldots,p_n)}(A)$ of
$A$ is defined as the norm (p-norm if $0<p<1$) of
$\widehat{A}$ as an operator from $\ell_{p_{1}}^{w}(E_{1})\times\ell_{p_{2}%
}^{w}(E_{2})\times\cdots\times\ell_{p_{n}}^{w}(E_{n})$ to $\ell_{p}^{w}(F)$.
\end{definition}



An example of coincidence situation involving weakly summing operators can be
found in \cite[Proposition 13]{popaI}: if $1<p,p_{1},...,p_{n}<\infty$ are such
that%
\[
\frac{1}{p^{\prime}}=\frac{1}{p_{1}^{\prime}}+\cdots+\frac{1}{p_{n}^{\prime}}%
\]
then%
\[
\mathcal{L}(E_{1},\ldots,E_{n};F)=\Pi_{w(p;p_{1},\ldots,p_{n})}(E_{1}%
,\ldots,E_{n};F).
\]

\begin{proposition}
\label{wascaract} Let $n\in{\mathbb{N}}$ and $0<p,p_{1},...,p_{n}<\infty$ be such that
$\frac{1}{p_{1}}+\frac{1}{p_{2}}+\cdots+\frac{1}{p_{n}}\geq\frac{1}{p}$, and
let $E_{1},\ldots,E_{n},F$ be Banach spaces. Then the spaces
\[
\Pi_{w(p;p_{1},\ldots,p_{n})}(E_{1},\ldots,E_{n};F^{\prime})
\]
and%
\[
\mathcal{L}(F,\Pi_{(p;p_{1},\ldots,p_{n})}(E_{1},\ldots,E_{n};\mathbb{K}))
\]
are isometrically isomorphic.
\end{proposition}

\begin{proof}
Given $A\in\Pi_{w(p;p_{1},\ldots,p_{n})}(E_{1},\ldots,E_{n};F^{\prime})$,
consider the linear operator
\[
\Phi_{A}\colon F\longrightarrow\Pi_{(p;p_{1},\ldots,p_{n})}(E_{1},\ldots
,E_{n};\mathbb{K})~,~\Phi_{A}(y)(x_{1},\ldots,x_{n})=\langle A(x_{1}%
,\ldots,x_{n}),y\rangle,
\]
for any $x_{k}\in E_{k}$, $k=1,\ldots,n$, and $y\in F$. Let $(x_{j}^{k}%
)_{j=1}^{\infty}\in\ell_{p_{k}}^{w}(E_{k}),$ $k=1,\ldots,n,$ and $y\in F$.
Then
\begin{align*}
\Vert(\Phi_{A}(y)(x_{j}^{1},\ldots,x_{j}^{n}))_{j}\Vert_{\ell_{p}} &  =\left(
\sum_{j}\left\vert \langle A(x_{j}^{1},\ldots,x_{j}^{n}),y\rangle\right\vert
^{p}\right)  ^{1/p}\\
&  \leq\pi_{w(p;p_{1},\ldots,p_{n})}(A)\cdot\Vert y\Vert\cdot\left\Vert
(x_{j}^{1})_{j=1}^{\infty}\right\Vert _{\ell_{p_{1}}^{w}(E_{1})}%
\cdots\left\Vert (x_{j}^{n})_{j=1}^{\infty}\right\Vert _{\ell_{p_{n}}%
^{w}(E_{n})}.
\end{align*}
Hence $\pi_{(p;p_{1},\ldots,p_{n})}(\Phi_{A}(y))\leq\pi_{w(p;p_{1}%
,\ldots,p_{n})}(A)\Vert y\Vert$ and therefore $\Vert\Phi_{A}\Vert\leq
\pi_{w(p;p_{1},\ldots,p_{n})}(A).$ Let now $\Phi\colon F\longrightarrow
\Pi_{(p;p_{1},\ldots,p_{n})}(E_{1},\ldots,E_{n};\mathbb{K})$ be given.
Consider the multilinear mapping
\[
A_{\Phi}\colon E_{1}\times\cdots\times E_{n}\longrightarrow F^{\prime
}~,~\langle A_{\Phi}(x_{1},\ldots,x_{n}),y\rangle=\Phi(y)(x_{1},\ldots,x_{n}),
\]
for any $x_{k}\in E_{k}$, $k=1,\cdots,n$, and $y\in F$. From
\begin{align*}
\Vert(A_{\Phi}(x_{j}^{1},\ldots,x_{j}^{n}))_{j}\Vert_{\ell_{p}^{w}(F^{\prime
})} &  =\sup_{\varphi\in B_{F^{\prime\prime}}}\left(  \sum_{j}\left\vert
\langle A_{\Phi}(x_{j}^{1},\ldots,x_{j}^{n}),\varphi\rangle\right\vert
^{p}\right)  ^{1/p}\\
&  =\sup_{y\in B_{F}}\left(  \sum_{j}\left\vert \Phi(y)(x_{j}^{1},\ldots
,x_{j}^{n})\right\vert ^{p}\right)  ^{1/p}\\
&  \leq\sup_{y\in B_{F}}\pi_{(p;p_{1},\ldots,p_{n})}(\Phi(y))\left\Vert
(x_{j}^{1})_{j=1}^{\infty}\right\Vert _{\ell_{p_{1}}^{w}(E_{1})}%
\cdots\left\Vert (x_{j}^{n})_{j=1}^{\infty}\right\Vert _{\ell_{p_{n}}%
^{w}(E_{n})}\\
&  {=}\Vert\Phi\Vert\cdot\left\Vert (x_{j}^{1})_{j=1}^{\infty}\right\Vert
_{\ell_{p_{1}}^{w}(E_{1})}\cdots\left\Vert (x_{j}^{n})_{j=1}^{\infty
}\right\Vert _{\ell_{p_{n}}^{w}(E_{n})},
\end{align*}
we conclude that $\pi_{w(p;p_{1},\ldots,p_{n})}(A_{\Phi})\leq\Vert\Phi\Vert$.
\end{proof}

In general, for $\frac{1}{p}\leq\frac{1}{p_{1}}+\frac{1}{p_{2}}+\cdots
+\frac{1}{p_{n}}<\frac{1}{p}+n-1$ one has that $\Pi_{w(p;p_{1},\ldots,p_{n}%
)}(E_{1},\ldots,E_{n};F)\subsetneq\mathcal{L}(E_{1},\ldots,E_{n};F)$, as shows
the example at the beginning of the section, where $n=2$, $E_{1}=E_{2}%
=\ell_{2}$ and $F=\mathbb{K}$.

To analyze when
\[
{\mathcal{L}}(E_{1},E_{2},\ldots,E_{n};F)=\Pi_{w(p;p_{1},p_{2},\ldots,p_{n}%
)}(E_{1},E_{2},\ldots,E_{n};F),
\]
for some values of $0\leq1/p_{1}+\cdots+1/p_{n}-1/p<n-1$ and some Banach
spaces $E_{1},\ldots,E_{n},F$, the consideration of projective tensor products
is quite profitable for our purposes. Recall that by $E_{1}\hat{\otimes}_{\pi
}\cdots\hat{\otimes}_{\pi}E_{n}$ we mean the completed projective tensor
product of the Banach spaces $E_{1},\ldots,E_{n}$. The following lemma will be
useful along this paper; part of it can be essentially found in \cite{soares}.

\begin{lemma}
\label{was} Let $n\in{\mathbb{N}}$ and $0<p,p_{1},...,p_{n}<\infty$ be such that $\frac{1}%
{p}\leq\frac{1}{p_{1}}+\frac{1}{p_{2}}+\cdots+\frac{1}{p_{n}}<\frac{1}{p}+n-1$,
and let $E_{1},\ldots,E_{n}$ be Banach spaces. The following are
equivalent:\newline
{\rm (i)} $\mathcal{L}(E_{1},\ldots,E_{n};F)=\Pi
_{w(p;p_{1},\ldots,p_{n})}(E_{1},\ldots,E_{n};F)$ and $\pi_{w(p;p_{1}%
,\ldots,p_{n})}\leq C\Vert\cdot\Vert$ for every Banach space $F$%
.\newline
{\rm (ii)} There exists $C>0$ such that
\[
\Vert(x_{j}^{1}\otimes\cdots\otimes x_{j}^{n})_{j=1}^{\infty}\Vert_{\ell
_{p}^{w}(E_{1}\hat{\otimes}_{\pi}\cdots\hat{\otimes}_{\pi}E_{n})}\leq
C\prod_{i=1}^{n}\Vert(x_{j}^{i})_{j=1}^{\infty}\Vert_{\ell_{p_{i}}^{w}(E_{i}%
)}.
\]
for all sequences $(x_{j}^{i})_{j=1}^{\infty}\in\ell_{p_{i}}^{w}(E_{i})$,
$i=1,\ldots,n$.\newline
{\rm (iii)} $\mathcal{L}(E_{1},\ldots,E_{n}%
;\mathbb{K})=\Pi_{(p;p_{1},\ldots,p_{n})}(E_{1},\ldots,E_{n};\mathbb{K})$ and
and there exists $C>0$ such that $\pi_{(p;p_{1},\ldots,p_{n})}\leq C\Vert
\cdot\Vert.$\newline
\end{lemma}

\begin{proof}
\noindent(i) $\Longrightarrow$ (ii) Take $F=E_{1}\hat{\otimes}_{\pi}\cdots
\hat{\otimes}_{\pi}E_{n}$ and $A\colon E_{1}\times\cdots\times E_{n}%
\longrightarrow E_{1}\hat{\otimes}_{\pi}\cdots\hat{\otimes}_{\pi}E_{n}$ given
by $A(x_{1},\ldots,x_{n})=x_{1}\otimes\cdots\otimes x_{n}$.

 \medskip

\noindent(ii) $\Longrightarrow$ (iii) Given $A\in\mathcal{L}(E_{1}%
,\ldots,E_{n};\mathbb{K})$, its linearization $T\colon E_{1}\hat{\otimes}%
_{\pi}\cdots\hat{\otimes}_{\pi}E_{n}\longrightarrow\mathbb{K}$ is bounded and
then $\widehat{T}\colon\ell_{p}^{w}(E_{1}\hat{\otimes}_{\pi}\cdots\hat
{\otimes}_{\pi}E_{n})\longrightarrow\ell_{p}$ is bounded. Now, given sequences
$(x_{j}^{i})_{j=1}^{\infty}\in\ell_{p_{i}}^{w}(E_{i})$, $i=1,\ldots,n$,
\begin{align*}
\Vert(A(x_{j}^{1},\ldots,x_{j}^{n}))_{j=1}^{\infty}\Vert_{p} &  =\Vert
(T(x_{j}^{1}\otimes\cdots\otimes x_{j}^{n}))_{j=1}^{\infty}\Vert_{p}%
=\Vert\widehat{T}((x_{j}^{1}\otimes\cdots\otimes x_{j}^{n})_{j=1}^{\infty
})\Vert_{p}\\
&  \leq\Vert\widehat{T}\Vert\cdot\Vert(x_{j}^{1}\otimes\cdots\otimes x_{j}%
^{n})_{j=1}^{\infty}\Vert_{\ell_{p}^{w}(E_{1}\hat{\otimes}_{\pi}\cdots
\hat{\otimes}_{\pi}E_{n})}\\
&  \leq C\Vert T\Vert\cdot\prod_{i=1}^{n}\Vert(x_{j}^{i})_{j=1}^{\infty}%
\Vert_{\ell_{p_{i}}^{w}(E_{i})}.
\end{align*}

\medskip\noindent(iii) $\Longrightarrow$ (i) For $T\in\mathcal{L}(E_{1}%
,\ldots,E_{n};F)$ and $\varphi\in F^{\prime}$ we have $\varphi\circ
T\in\mathcal{L}(E_{1},\ldots,E_{n};\mathbb{K})=\Pi_{(p;p_{1},\ldots,p_{n}%
)}(E_{1},\ldots,E_{n};\mathbb{K})$ and the result follows straightforwardly.
\end{proof}

A direct consequence of the previous lemma and the Inclusion Theorem is that
if $p,p_{k}\geq1$ for $k=1,\ldots,n$ are such that $\frac{1}{p_{1}}+\frac
{1}{p_{2}}+\cdots+\frac{1}{p_{n}}-n+1\geq\frac{1}{p},$ then
\[
\mathcal{L}(E_{1},\ldots,E_{n};F)=\Pi_{w(p;p_{1},\ldots,p_{n})}(E_{1}%
,\ldots,E_{n};F)
\]
for all $E_{1},...,E_{n},F.$ So, weakly $(p;p_{1},\ldots,p_{n})$-summing
operators are a variant of the concept of $(p;p_{1},\ldots,p_{n})$-summing
operators that shed light on the summability properties for operators. Let us
introduce another variant of the notion of summing multilinear operator (for a
related notion we refer to \cite{achour}).

\begin{definition}\rm
{\label{def2}Let }$1\leq p,p_1,\ldots,p%
_n<\infty$  and $\frac{1}{p_1}+\frac{1}{p_2%
}+\cdots+\frac{1}{p_n}\geq\frac{1}{p}.$ We say that $A\in\mathcal{L}%
(E_1,\ldots,E_n;F)$ is {\it Cohen $(p;p_1,\ldots,p_n)$-nuclear} if
the induced mapping $\widehat{A}\colon\ell_{p_{1}}^{w}(E_{1})\times\ell
_{p_{2}}^{w}(E_{2})\times\cdots\times\ell_{p_{n}}^{w}(E_{n})\longrightarrow
\ell_{p}\langle F\rangle$ is well-defined (hence $n$-linear and bounded). The
space formed by these mappings is denoted by $CN_{(p;p_{1},\ldots,p_{n}%
)}(E_{1},\ldots,E_{n};F)$ and the Cohen $(p;p_{1},\ldots,p_{n})$-nuclear norm
$\Vert A\Vert_{CN(p;p_{1},\ldots,p_{n})}$ is defined as the norm of
$\widehat{A}$ as operator from $\ell_{p_{1}}^{w}(E_{1})\times\ell_{p_{2}}%
^{w}(E_{2})\times\cdots\times\ell_{p_{n}}^{w}(E_{n})$ to $\ell_{p}\langle
F\rangle$.
\end{definition}

Of course
\[
CN_{(p;p_{1},\ldots,p_{n})}(E_{1},\ldots,E_{n};F)\subseteq\Pi_{(p;p_{1}%
,\ldots,p_{n})}(E_{1},\ldots,E_{n};F)\subseteq\Pi_{w(p;p_{1},\ldots,p_{n}%
)}(E_{1},\ldots,E_{n};F).
\]

\begin{theorem}
\label{cncaract} Let $n\in{\mathbb{N}}$ and $1\leq p,p_{1},...,p_{n}<\infty$ be such that
$\frac{1}{p_{1}}+\frac{1}{p_{2}}+\cdots+\frac{1}{p_{n}}\geq\frac{1}{p}$, and
let $E_{1},\ldots,E_{n}$ and $F$ be Banach spaces. Then the spaces
\[
CN_{(p;p_{1},\ldots,p_{n})}(E_{1},\ldots,E_{n};F^{\prime})
\]
and
\[
\Pi_{(1;p_{1},\ldots,p_{n},p^{\prime})}(E_{1},\ldots,E_{n},F;\mathbb{K})
\]
are isometrically isomorphic.
\end{theorem}

\begin{proof}
Clearly one has that%
\[
\mathcal{L}(E_{1},\ldots,E_{n},F;\mathbb{K})=\mathcal{L}(E_{1},\ldots
,E_{n};F^{\prime})
\]
with the identification $A\longleftrightarrow\tilde{A}$ given by
\[
A(x_{1},\ldots,x_{n},z)=\langle\tilde{A}(x_{1},\ldots,x_{n}),z\rangle
\]
for $x_{k}\in E_{k}$, $k=1,\ldots,n$ and $z\in F$. Given $\tilde{A}\in
CN_{(p;p_{1},\ldots,p_{n})}(E_{1},\ldots,E_{n},F^{\prime})$, take $(x_{j}%
^{k})_{j=1}^{\infty}\in\ell_{p_{k}}^{w}(E_{k}),$ $k=1,\ldots,n,$ and
$(z_{j})_{j=1}^{\infty}\in\ell_{p^{\prime}}^{w}(F).$ Then%
\begin{align*}
\sum_{j}\left\vert A(x_{j}^{1},\ldots,x_{j}^{n},z_{j})\right\vert =  &
\sum_{j}\left\vert \langle\tilde{A}(x_{j}^{1},\ldots,x_{j}^{n}),z_{j}%
\rangle\right\vert \\
&  \leq\Vert(\tilde{A}(x_{j}^{1},\ldots,x_{j}^{n}))_{j}\Vert_{\ell_{p}\langle
F^{\prime}\rangle}\cdot\Vert(z_{j})_{j=1}^{\infty}\Vert_{\ell_{p^{\prime}}%
^{w}(F)}\\
&  \leq\Vert\tilde{A}\Vert_{CN(p;p_{1},\ldots,p_{n})}\cdot\Vert(x_{j}%
^{1})_{j=1}^{\infty}\Vert_{\ell_{p_{1}}^{w}(E_{1})}\cdots\Vert(x_{j}%
^{n})_{j=1}^{\infty}\Vert_{\ell_{p_{n}}^{w}(E_{n})}\cdot\Vert(z_{j}%
)_{j=1}^{\infty}\Vert_{\ell_{p^{\prime}}^{w}(F)}.
\end{align*}
Hence $A\in\Pi_{(1;p_{1},\ldots,p_{n},p^{\prime})}(E_{1},\ldots,E_{n}%
,F;\mathbb{K})$ and $\pi_{(1;p_{1},\ldots,p_{n},p^{\prime})}(A)\leq\Vert
\tilde{A}\Vert_{CN(p;p_{1},\ldots,p_{n})}.$ Assume now that $A\in\Pi
_{(1;p_{1},\ldots,p_{n},p^{\prime})}(E_{1},\ldots,E_{n},F;\mathbb{K})$ and let
$(x_{j}^{k})_{j=1}^{\infty}\in\ell_{p_{k}}^{w}(E_{k}),$ $k=1,\ldots,n,$ and
$(z_{j})_{j=1}^{\infty}\in\ell_{p^{\prime}}^{w}(F)$ be given. Then%
\begin{align*}
\sum_{j}\left\vert \langle\tilde{A}(x_{j}^{1},\ldots,x_{j}^{n}),z_{j}%
\rangle\right\vert =  &  \sum_{j}\left\vert A(x_{j}^{1},\ldots,x_{j}^{n}%
,z_{j})\right\vert \\
&  \leq\pi_{(1;p_{1},\ldots,p_{n},p^{\prime})}(A)\cdot\Vert(x_{j}^{1}%
)_{j=1}^{\infty}\Vert_{\ell_{p_{1}}^{w}(E_{1})}\cdots\Vert(x_{j}^{n}%
)_{j=1}^{\infty}\Vert_{\ell_{p_{n}}^{w}(E_{n})}\cdot\Vert(z_{j})_{j=1}%
^{\infty}\Vert_{\ell_{p^{\prime}}^{w}(F)}.
\end{align*}
Hence $\tilde{A}\in CN_{(p;p_{1},\ldots,p_{n})}(E_{1},\ldots,E_{n};F^{\prime
})$ and $\Vert\tilde{A}\Vert_{CN(p;p_{1},\ldots,p_{n})}\leq\pi_{(1;p_{1}%
,\ldots,p_{n},p^{\prime})}(A).$
\end{proof}

\begin{corollary}
Let $n\in{\mathbb{N}}$, $1\leq p,p_{1},...,p_{n}<\infty,$ $\frac{1}{p_{1}%
}+\frac{1}{p_{2}}+\cdots+\frac{1}{p_{n}}\geq1$, and let $E_{1},\ldots,E_{n}$
and $F$ be Banach spaces. Then
\[
{\mathcal{L}}(F;CN_{(p_{n}^{\prime};p_{1},\ldots,p_{n-1})}(E_{1},\ldots
,E_{n};E_{n}^{\prime}))
\]
and%
\[
\Pi_{w(1;p_{1},\ldots,p_{n})}(E_{1},\ldots,E_{n};F^{\prime})
\]
are isometrically isomorphic.
\end{corollary}

\begin{proof}
It follows from  Theorem \ref{cncaract} 
and  Proposition \ref{wascaract}.
\end{proof}

\section{Coincidences between spaces of summing multilinear mappings}

\label{lift}

It is well known that cotype plays a fundamental role in coincidence results
for linear and nonlinear operators. For example, if $E$ has cotype $2$ then
$\Pi_{(1,1)}(E;F)=\Pi_{(2,2)}(E;F)$ for any Banach space $F$ \cite[Corollary
11.16(a)]{Diestel}.
The first attempt to lift this result to the multilinear setting yields the
following result (see \cite[Corollary 3.7(a)(i)]{pop}). The polynomial version appeared in \cite[Theorem 16]{MT}).

\begin{proposition}{\rm \cite[Corollary 3.7(a)(i)]{pop}}
Let $E_{1}, \ldots, E_{n}$ be cotype 2 spaces. Then
\[
\Pi_{(\frac{1}{n};1,\ldots,1)}(E_{1},\ldots,E_{n};F) = \Pi_{(\frac{2}%
{n};2,\ldots,2)}(E_{1},\ldots,E_{n};F)
\]
for every Banach space $F$.
\end{proposition}

\begin{corollary}
Let $E_{1}, \ldots, E_{n}$ be cotype 2 spaces. Then
\[
\Pi_{w(\frac{1}{n};1,\ldots,1)}(E_{1},\ldots,E_{n};F) = \Pi_{w(\frac{2}%
{n};2,\ldots,2)}(E_{1},\ldots,E_{n};F)
\]
for every Banach space $F$.
\end{corollary}

Recently, in \cite[Theorem 3 and Remark 2]{Junek}, \cite[Corollary 4.6]{pop}
and \cite[Theorem 3.8 (ii)]{michels} it was proved that if $E_{1},\ldots
,E_{n}$ are Banach spaces with cotype $c$ and $n\geq2$ then:

(i) If $c=$ $2,$ then
\begin{equation}
\Pi_{\left(  q;q,\ldots,q\right)  }(E_{1},\ldots,E_{n};F)\subseteq\Pi_{\left(
p;p,\ldots,p\right)  }(E_{1},\ldots,E_{n};F) \label{nnbb}%
\end{equation}
holds for $1\leq p\leq q\leq2$ and arbitrary Banach spaces $F$.

(ii) If $c>2,$ then
\begin{equation}
\Pi_{\left(  q;q,\ldots,q\right)  }(E_{1},\ldots,E_{n};F)\subseteq\Pi_{\left(
p;p,\ldots,p\right)  }(E_{1},\ldots,E_{n};F) \label{nnbb2}%
\end{equation}
holds for $1\leq p\leq q<c^{\prime}$ and arbitrary Banach spaces $F$.

It is easy to see that the inclusions (\ref{nnbb}) and (\ref{nnbb2}) are not
optimal. So it is natural to ask for the best $s$ for which, under the same
assumptions,%
\[
\Pi_{\left(  q;q,\ldots,q\right)  }(E_{1},\ldots,E_{n};F)\subseteq\Pi_{\left(
s;p,\ldots,p\right)  }(E_{1},\ldots,E_{n};F).
\]
Let us settle this question and get a much better result in this direction.

We aim to prove a more general result. The key point is to work with spaces
for which we have the factorization $\ell_{p}^{w}(E)=\ell_{r} \ell_{s}^{w}(E)$
for $1/p=1/r+1/s$.

\begin{lemma}
\label{nuevolema} Let $1< p<r<\infty$ and let $E$ be a Banach space.\newline%
{\rm (a)} If $\mathcal{L}(\ell_{p^{\prime}};E)=\Pi_{(r;r)}(\ell_{p^{\prime
}};E)$, then $\ell_{p}^{w}(E)=\ell_{r}\ell_{s}^{w}(E)$, where $1/r+1/s=1/p$%
.\newline
{\rm (b)} If $\mathcal{L}(c_{0};E)=\Pi_{(r;r)}(c_{0};E)$, then
$\ell_{1}^{w}(E)=\ell_{r}\ell_{s}^{w}(E)$, where $1/r+1/s=1$.
\end{lemma}

\begin{proof} Take $(x_j)_{j=1}^\infty\in \ell_p^w(E)$.
Define $u\colon \ell_{p'}\longrightarrow E$ by $u(e_j)=x_j$, $j \in \mathbb{N}$. Since $u\in \Pi_{(r;r)}(\ell_{p'};E)$, by \cite[Lemma 2.23]{Diestel} there exist $(\alpha_j)_{j=1}^\infty\in \ell_{r}$ and $(y_j)_{j=1}^\infty\in
\ell_s^w(E)$ such that $u(e_j)=x_j=\alpha_j y_j$ for every $j$. The other case is similar.
\end{proof}

It is well known that if $E$ has cotype $2$ then $\mathcal{L}(c_{0}%
;E)=\Pi_{2;2}(c_{0};E))$ (see \cite[page 224]{Diestel}). Hence for cotype $2$
space one has $\ell_{1}^{w}(E)=\ell_{2}\ell_{2}^{w}(E)$. In order to
understand the use of cotype in factorization, let us recall the following
definitions from \cite[Sections 16.4, 20.1]{Pietsch}: Let $E$ be a Banach
space, $0<p\leq s\leq\infty$ and $r$ be determined by $1/r+1/s=1/p$.

$\bullet$ A sequence $(x_{j})_{j=1}^{\infty}$ in $E$ is called \emph{mixed
$(s,p) $-summable} if it can be written in the form $x_{j}=\alpha_{j}y_{j}$
with $(\alpha_{j})_{j=1}^{\infty}\in\ell_{r}$ and $(y_{j})_{j=1}^{\infty}%
\in\ell^{w}_{s}(E)$.

$\bullet$ An operator $u \colon E\longrightarrow F$ is called \emph{$(s,p)$%
-mixing} if every weakly $p$-summable sequence in $E$ is mapped into an
$(s,p)$-mixed summable sequence in $F$. ${\mathcal{M}}_{(s,p)}$ denotes the
ideal of $(s,p)$-mixing operators.

From the definition, if $id_{E}$ is the identity operator on the Banach space
$E$, then $id_{E}\in{\mathcal{M}}_{(s,p)}(E;E)$ if and only if $\ell^{w}%
_{p}(E)=\ell_{r}\ell^{w}_{s}(E)$, where $1/r+1/s=1/p$. In \cite[Theorem 20.3.1
]{Pietsch} it is proved that, for $s\geq1$, ${\mathcal{M}}_{(s,p)}=\Pi
_{(s,s)}^{-1}\circ\Pi_{(p,p)}$ which relates, in a very strong way, mixing
operators with absolutely summing operators. Actually it says that the
identity $id_{E}$ is $(s,p)$-mixing if and only if $\Pi_{(s,s)}(E;F)=\Pi
_{(p,p)}(E;F)$ for all Banach space $F$. In other words, $\ell^{w}_{p}%
(E)=\ell_{r}\ell^{w}_{s}(E)$ if and only if $\Pi_{(s,s)}(E;F)=\Pi
_{(p,p)}(E;F)$ for all Banach space $F$.

Now, in order to find examples of Banach spaces $E$ for which $\ell^{w}%
_{p}(E)=\ell_{r}\ell^{w}_{s}(E)$, we look for cotype:

$\bullet$ If $E$ has cotype $2$ then $\Pi_{(s;s)}(E;F)=\Pi_{(p;p)}(E;F)$ for
every Banach space $F$, whenever $1\leq p\leq s\leq2$. In this case $\ell
^{w}_{p}(E)=\ell_{r}\ell^{w}_{s}(E)$.

$\bullet$ If $E$ has cotype $s$, $2<s<\infty$, then $\Pi_{(q;q)}%
(E;F)=\Pi_{(p;p)}(E;F)$ for every Banach space $F$, whenever $1\leq p\leq
q<s^{\prime}$. In this case $\ell^{w}_{p}(E)=\ell_{r}\ell^{w}_{q}(E)$.

Therefore one has the following result.

\begin{lemma}
\label{nuevolema} Let $E$ be a Banach space of cotype $2\leq s<\infty$, let
$p\leq q$ and $1/p=1/r+1/q$. Then\newline
{\rm (a)} $\ell_{p}^{w}%
(E)=\ell_{r}\ell_{q}^{w}(E)$ for $s=2$ and $1\leq p\leq q\leq2$.\newline%
{\rm(b)} $\ell_{p}^{w}(E)=\ell_{r}\ell_{q}^{w}(E)$ for $s>2$ and $1\leq
p\leq q<s^{\prime}$.
\end{lemma}

\begin{theorem}
\label{teornovo} For $i=1,\ldots,n$, let $E_{i}$ be a Banach space with cotype
$c_{i}\in\lbrack2,\infty]$. Let $1<p_{i},q_{i}<\infty$ with $1/r_{i}%
=1/p_{i}-1/q_{i}\geq0$, $p\leq q$ and $\frac{1}{p}=\sum_{i=1}^{n}\frac
{1}{r_{i}}+\frac{1}{q}$. Assume that
\[
1\leq p_{i}=q_{i}\text{ ~if }c_{i}=\infty,
\]%
\[
1\leq p_{i}\leq q_{i}\leq2\text{ ~if }c_{i}=2,
\]%
\[
1\leq p_{i}\leq q_{i}<c_{i}^{\prime}\text{ ~if }2<c_{i}<\infty.
\]
Then
\[
\Pi_{(q;q_{1},\ldots,q_{n})}(E_{1},\ldots,E_{n};F)=\Pi_{(p;p_{1},\ldots
,p_{n})}(E_{1},\ldots,E_{n};F),
\]
for every Banach space $F$. In particular, if each $c_{i}\in\lbrack2,\infty)$
and $1\leq p\leq q<\min\limits_{i}\{c_{i}^{\prime}\}$, then
\[
\Pi_{(q;q,\ldots,q)}(E_{1},\ldots,E_{n};F)=\Pi_{(\frac{qp}{n\left(
q-p\right)  +p};p,\ldots,p)}(E_{1},\ldots,E_{n};F)
\]
and
\[
\Pi_{(q;q,\ldots,q)}(E_{1},\ldots,E_{n};F)\subseteq\Pi_{(p;p,\ldots,p)}%
(E_{1},\ldots,E_{n};F).
\]

\end{theorem}

\begin{proof}
The inclusion \[
\Pi_{(p;p_1,\ldots,p_n)}(E_{1},\ldots,E_{n};F) \subseteq \Pi
_{(q;q_{1},\ldots,q_{n})}(E_{1},\ldots,E_{n};F)
\]
follows from the Inclusion Theorem \ref{inclusiontheorem}. Let us suppose that $A \in
\Pi_{(q;q_1,\ldots,q_n)}(E_{1},\ldots,E_{n};F)$ and let the sequences
$(x_j^i)_{j=1}^\infty \in \ell_{p_i}^w(E_i)$, $i = 1, \ldots, n$, be given. From Lemma \ref{nuevolema}
we know that $\ell_{p_i}^w(E_i) = \ell_{r_i} \ell_{q_i}^w(E_i)$, $i = 1, \ldots, n$. Hence there are sequences $(\alpha_j^i)_{j=1}^\infty \in \ell_{r_i}$ and $(y_j^i)_{j=1}^\infty
\in \ell_{q_i}^w(E_i)$ such that $(x_j^i)_{j=1}^\infty = (\alpha_j^i
y_j^i)_{j=1}^\infty$, $i = 1, \ldots, n$. In this fashion,
$(\alpha_j^1 \cdots \alpha_j^k)_{j=1}^\infty \in \ell_{r_1} \cdots
\ell_{r_n} = \ell_{r}$, where $1/r=\sum_{j=1}^n 1/r_j$, and $(A(y_j^1, \ldots,
y_j^n))_{j=1}^\infty \in \ell_{q}(F)$. Since
$\frac{1}{r} + \frac{1}{q}  =\frac{1}{p}$ it follows that
$$ (A(x_j^1, \ldots, x_j^n))_{j=1}^\infty =(\alpha_j^1 \cdots \alpha_j^n A(y_j^1, \ldots, y_j^n))_{j=1}^\infty
\in \ell_p(F).$$
\end{proof}

The above theorem is a variation of some well known results in this area
proved first by D. Popa in \cite[Theorem 1, Corollaries 2, 3 and 4]{popaarchiv}
and later by A. T. Bernardino \cite[Theorem 3.1]{thi} .


\section{Summability of all multilinear mappings}

\label{section3}

In Section \ref{D} we have observed some coincidence results for weakly
absolutely summing multilinear mappings. We are now interested in
understanding when
\[
{\mathcal{L}}(E_{1},E_{2},\ldots,E_{n};F)=\Pi_{(p;p_{1},p_{2},\ldots,p_{n}%
)}(E_{1},E_{2},\ldots,E_{n};F)
\]
for some values $p,p_{1},...,p_{n}$ and some Banach spaces $E_{1},\ldots
,E_{n},F$. First, we will show how to lift coincidence results from the
bilinear and trilinear case to the $n$-linear setting. Afterwards, we will
explore related notions as the cotype of a Banach space to get our purpose.

Let us now explain a procedure which allows to lift coincidence results of
bilinear maps to coincidence results of multilinear maps.

\begin{theorem}
\label{lifting}
Given $n \in\mathbb{N}$, $n \geq2$, let $m\in\mathbb{N}$ be such that $n = 2m$
if $n$ is even and $n = 2m +1$ if $n$ is odd. For $j=1,\ldots,n,$ let $E_{j}$
be a Banach space and let $1\leq r_{j}<\infty$, $p_{1},\ldots,p_{m}>0$ be such
that $1/p_{1}+\cdots+1/p_{m} > m-1$ and  $1/p_i\leq 1/r_{2i-1}+1/r_{2i}<1/p_i+1$, $i=1,\ldots,m$. Assume that
\[
\mathcal{L}(E_{2i-1},E_{2i};\mathbb{K})=\Pi_{(p_{i};r_{2i-1},r_{2i})}%
(E_{2i-1},E_{2i};\mathbb{K})
\]
for all $i=1,\ldots,m.$\newline
{\rm (i)} If $n$ is even, then
\[
\mathcal{L}(E_{1},\ldots,E_{n};\mathbb{K})=\Pi_{(p;r_{1},\ldots,r_{n})}%
(E_{1},\ldots,E_{n};\mathbb{K}),
\]
whenever $1/p\leq1/p_{1}+\cdots+1/p_{m} -m+1$. \newline
{\rm (ii)} If $n$ is
odd, then
\[
\mathcal{L}(E_{1},\ldots,E_{n};\mathbb{K})=\Pi_{(q;r_{1},\ldots,r_{n})}%
(E_{1},\ldots,E_{n};\mathbb{K}),
\]
whenever $1/p\leq1/p_{1}+\cdots+1/p_{m} -m+1$ and $1/q\leq1/r_{2m+1}+1/p-1.$
\end{theorem}

\begin{proof}
(i) Let $A\in\mathcal{L}(E_{1},\ldots,E_{n};\mathbb{K})=\mathcal{L}%
(E_{1},\ldots,E_{2m};\mathbb{K}).$ Using the associativity of the projective
tensor norm $\pi$ it is easy to see that there is an $m$-linear mapping
\[
B\in\mathcal{L}(E_{1}\hat{\otimes}_{\pi}E_{2},\ldots,E_{2m-1}\hat{\otimes
}_{\pi}E_{2m};\mathbb{K})
\]
such that
\[
B(x^{1}\otimes x^{2},\ldots,x^{2m-1}\otimes x^{2m})=A(x^{1},x^{2}%
,\ldots,x^{2m-1},x^{2m}),
\]
for all $x^{j}\in E_{j}$. Using the Defant-Voigt Theorem we know that $B$ is
$\left(  1;1,...,1\right)  $-summing and, since $1/p\leq1/p_{1}+\cdots
+1/p_{m}-m+1$, the inclusion theorem (Theorem \ref{inclusiontheorem}) tells us
that
\[
B\in\Pi_{(p;p_{1},\ldots,p_{m})}(E_{1}\hat{\otimes}_{\pi}E_{2},\ldots
,E_{2m-1}\hat{\otimes}_{\pi}E_{2m};\mathbb{K}).
\]
From Lemma \ref{was} it follows that
\begin{align*}
\lefteqn{\left(  \sum_{j=1}^{\infty}\left\vert A(x_{j}^{1},\ldots,x_{j}%
^{2m})\right\vert ^{p}\right)  ^{1/p}=\left(  \sum_{j=1}^{\infty}\left\vert
B(x_{j}^{1}\otimes x_{j}^{2},\ldots,x_{j}^{2m-1}\otimes x_{j}^{2m})\right\vert
^{p}\right)  ^{1/p}} &  & \\
&  &  &  \leq\pi_{(p;p_{1},\ldots,p_{m})}(B)\left\Vert (x_{j}^{1}\otimes
x_{j}^{2})_{j=1}^{\infty}\right\Vert _{\ell_{p_{1}}^{w}(E_{1}\hat{\otimes
}_{\pi}E_{2})}\cdots\left\Vert (x_{j}^{2m-1}\otimes x_{j}^{2m})_{j=1}^{\infty
}\right\Vert _{\ell_{p_{m}}^{w}(E_{2m-1}\hat{\otimes}_{\pi}E_{2m})}\\
&  &  &  \leq C\left\Vert B\right\Vert \cdot\prod_{i=1}^{2m}\left\Vert
(x_{j}^{i})_{j=1}^{\infty}\right\Vert _{\ell_{r_{i}}^{w}(E_{i})}%
\end{align*}
for all $(x_{j}^{i})_{j=1}^{\infty}\in\ell_{r_{i}}^{w}(E_{i})$, $i=1,\ldots
,2m=n$. Then $A\in\Pi_{(p;r_{1},\ldots,r_{2m})}(E_{1},\ldots,E_{2m}%
;\mathbb{K})$.

\bigskip
\noindent(ii) Let $A\in\mathcal{L}(E_{1},\ldots,E_{n};\mathbb{K}%
)=\mathcal{L}(E_{1},\ldots,E_{2m},E_{2m+1};\mathbb{K}).$ From (i) we know
that
\[
\mathcal{L}(E_{1},\ldots,E_{2m};\mathbb{K})=\Pi_{(p;r_{1},\ldots,r_{2m}%
)}(E_{1},\ldots,E_{2m};\mathbb{K}).
\]
Then, by \cite[Corollary 3.2]{Port} we get that
\[
\mathcal{L}(E_{1},\ldots,E_{2m+1};\mathbb{K})=\Pi_{(p;r_{1},\ldots,r_{2m}%
,1)}(E_{1},\ldots,E_{2m+1};\mathbb{K}).
\]
Hence, $A\in\Pi_{(p;r_{1},\ldots,r_{2m},1)}(E_{1},\ldots,E_{2m+1}%
;\mathbb{K}).$ Using the Inclusion Theorem \ref{inclusiontheorem} once again
we conclude that $A\in\Pi_{(q;r_{1},\ldots,r_{2m},r)}(E_{1},\ldots
,E_{2m+1};\mathbb{K})$ for any $p\leq q$ and $1\leq r<\infty$ such that
$1/p-1/q\geq1/r^{\prime}$. Choose $r=r_{2m+1}$ to complete the proof.
\end{proof}

Applying the previous result to the case $p_{i}=1$, $E_{i}=E$ and $r_{i}%
=r\geq1$ for any values of $i$, one obtains that if $\mathcal{L}%
(^{2}E;\mathbb{K})=\Pi_{(1;r,r)}(^{2}E;\mathbb{K})$ then
\[
\mathcal{L}(^{2n}E;\mathbb{K})=\Pi_{(1;r,\ldots,r)}(^{2n}E;\mathbb{K})
\mathrm{~~and~~} \mathcal{L}(^{2n+1}E;\mathbb{K})=\Pi_{(r;r,\ldots,r)}%
(^{2n+1}E;\mathbb{K}).
\]

We can actually improve a bit the result by imposing conditions on trilinear maps.

\begin{theorem}
Let $1\leq r\leq2$ and $E$ be a Banach space. If $\mathcal{L}(^{2}%
E;\mathbb{K})=\Pi_{(1;r,r)}(^{2}E;\mathbb{K})$ and $\mathcal{L}(^{3}%
E;\mathbb{K})=\Pi_{(1;r,r,r)}(^{3}E;\mathbb{K})$, then%
\[
\mathcal{L}(^{n}E;\mathbb{K})=\Pi_{(1;r,\ldots,r)}(^{n}E;\mathbb{K})
\]
for every $n\geq2.$
\end{theorem}

\begin{proof}
We have already proved the case $n$ even. Let us consider the case $n$ odd and
proceed by induction. Suppose that the result is valid for a fixed $k$ odd.
Let us prove that it is also true for $k+2$. Given $A\in\mathcal{L}%
(^{k+2}E;\mathbb{K})$, let $F=E\hat{\otimes}_{\pi}\cdots\hat{\otimes}_{\pi}E$
($k$ times, that is $F=\hat{\otimes}_{\pi}^{k}E$) and $G=E\hat{\otimes}_{\pi
}E.$ From the associativity properties of the projective norm there is a
bilinear form $B\in\mathcal{L}(F,G;\mathbb{K})$ so that
\[
B(x^{1}\otimes\cdots\otimes x^{k},x^{k+1}\otimes x^{k+2})=A(x^{1}%
,\ldots,x^{k+2}),
\]
for all $x^{j}\in E$, $j=1,\ldots,k+2$. Using the Defant-Voigt Theorem we know
that $B$ is $\left(  1;1,1\right)  $-summing and
\begin{align*}
\sum_{j=1}^{\infty}\left\vert A(x_{j}^{1},\ldots,x_{j}^{k+2})\right\vert  &
=\sum_{j=1}^{\infty}\left\vert B(x_{j}^{1}\otimes\cdots\otimes x_{j}^{k}%
,x_{j}^{k+1}\otimes x_{j}^{k+2})\right\vert \\
&  \leq\pi_{(1;1,1)}(B)\cdot\left\Vert (x_{j}^{1}\otimes\ldots\otimes
x_{j}^{k})_{j=1}^{\infty}\right\Vert _{\ell_{1}^{w}(E\hat{\otimes}_{\pi}%
\cdots\hat{\otimes}_{\pi}E)}\cdot\left\Vert (x_{j}^{k+1}\otimes x_{j}%
^{k+2})_{j=1}^{\infty}\right\Vert _{\ell_{1}^{w}(E\hat{\otimes}_{\pi}E)}.
\end{align*}
From Lemma \ref{was}, there are positive constants $C_{1}$ and $C_{2}$ such
that
\begin{align*}
&  \left\Vert (x_{j}^{1}\otimes\ldots\otimes x_{j}^{k})_{j=1}^{\infty
}\right\Vert _{\ell_{1}^{w}(E\hat{\otimes}_{\pi}\cdots\hat{\otimes}_{\pi}%
E)}\cdot\left\Vert (x_{j}^{k+1}\otimes x_{j}^{k+2})_{j=1}^{\infty}\right\Vert
_{\ell_{1}^{w}(E\hat{\otimes}_{\pi}E)}\\
&  \leq\left(  C_{1}\left\Vert (x_{j}^{1})_{j=1}^{\infty}\right\Vert
_{\ell_{r}^{w}(E)}\cdots\left\Vert (x_{j}^{k})_{j=1}^{\infty}\right\Vert
_{\ell_{r}^{w}(E)}\right)  \left(  C_{2}\left\Vert (x_{j}^{k+1})_{j=1}%
^{\infty}\right\Vert _{\ell_{r}^{w}(E)}\cdot\left\Vert (x_{j}^{k+2}%
)_{j=1}^{\infty}\right\Vert _{\ell_{r}^{w}(E_{{}})}\right)
\end{align*}
for all $(x_{j}^{i})_{j=1}^{\infty}\in\ell_{r}^{w}(E)$, $i=1,\ldots,k+2$.
\end{proof}

Next we present a technique to lift $(n-1)$-linear coincidences to $n$-linear
coincidences. A few definitions are in order: by $Rad(E)$ we denote the space
of sequences $(x_{j})_{j=1}^{\infty}$ in $E$ such that
\[
\Vert(x_{j})_{j=1}^{\infty}\Vert_{Rad(E)}:=\sup_{n\in\mathbb{N}}\left\Vert
\sum_{j=1}^{n}r_{j}x_{j}\right\Vert _{L^{2}([0,1],E)}<\infty,
\]
where $(r_{j})_{j\in\mathbb{N}}$ are the usual Rademacher functions.

Let $\ell_{p}^{u}(E)$ denote the closed subspace of $\ell_{p}^{w}(E)$ formed
by the sequences $(x_{j})_{j=1}^{\infty}\in\ell_{p}^{w}(E)$ such that
$\lim_{k\rightarrow\infty}\Vert(x_{j})_{j=k}^{\infty}\Vert_{\ell_{p}^{w}%
(E)}=0.$ It is well known that $\ell_{p_{k}}^{w}(E_{k})$ can be replaced with
$\ell_{p_{k}}^{u}(E_{k})$ in the definition of absolutely summing mappings (see, e.g., \cite{acad}).

According to \cite{Nach, Archiv} (see \cite[Chapter 12]{Diestel} for the
linear case), a multilinear map $A\in{\mathcal{L}}(E_{1},\ldots,E_{n};F)$ is
said to be \textit{almost summing} if the induced map $\widehat{A}\colon
\ell_{2}^{u}(E_{1}) \times\cdots\times\ell_{2}^{u}(E_{n}) \longrightarrow
Rad(F)$ is well-defined (hence $n$-linear and bounded).
We write $\Pi_{as}(E_{1},\ldots,E_{n};F)$ for the space of all such
multilinear maps.

Remember that a \textit{GT-space} is a Banach space $E$ for which every
bounded linear operator from $E$ to $\ell_{2}$ is absolutely 1-summing. The
following result is a variant of \cite[Theorem~3.7]{BBPR} and the same
technique was used in \cite{Port}:

\begin{proposition}
Let $n\ge2$ and let $E_{1}, \ldots, E_{n}$ be Banach spaces such that
\[
\mathcal{L}(E_{1},\ldots,E_{n-1};E_{n}^{\prime})= \Pi_{as}(E_{1}%
,\ldots,E_{n-1};E_{n}^{\prime}).
\]
{\rm (i)} Then
\[
\mathcal{L}(E_{1},\ldots,E_{n};\mathbb{K})= \Pi_{(1;2,\ldots,2,1)}%
(E_{1},\ldots,E_{n};\mathbb{K}).
\]
{\rm (ii)} If $E_{n}^{\prime}$ is a $GT$-space of cotype 2 then
\[
\mathcal{L}(E_{1},\ldots,E_{n};\mathbb{K})= \Pi_{(1;2,\ldots,2)}(E_{1}%
,\ldots,E_{n};\mathbb{K}).
\]

\end{proposition}

\begin{proof}
Given $A\in\mathcal{L}(E_{1},\ldots,E_{n};\mathbb{K})$, define $A_{n-1}%
\in\mathcal{L}(E_{1},\ldots,E_{n-1};E_{n}^{\prime})$ in the obvious way, that
is,
\[
A_{n-1}(x_{1},\ldots,x_{n-1})(x_{n})=A(x_{1},\ldots,x_{n}).
\]
By assumption, $A_{n-1}\in\Pi_{as}(E_{1},\ldots,E_{n-1};E_{n}^{\prime})$.
Let $(x_{j}^{i})_{j\in\mathbb{N}}\in\ell_{2}^{u}(E_{i})$ for $1\leq i\leq n-1$
and $(x_{j}^{n})_{j\in\mathbb{N}}\in\ell_{1}^{u}(E_{n})$ be given. For any
$m\in\mathbb{N}$ there exists $(\lambda_{j})_{j=1}^{m}$ such that
$|\lambda_{j}|=1$ and we have
\[
\sum_{j=1}^{m}|A(x_{j}^{1},\ldots,x_{j}^{n})|=\sum_{j=1}^{m}A_{n-1}(x_{j}%
^{1},\ldots,x_{j}^{n-1})(\lambda_{j}x_{j}^{n})\hspace*{20em}%
\]
\vspace*{-0.9em}
\[
=\int_{0}^{1}\sum_{j=1}^{m}\left[  (A_{n-1}(x_{j}^{1},\ldots,x_{j}^{n-1}%
)r_{j}(t)\right]  \left(  \sum_{j=1}^{m}\lambda_{j}x_{j}^{n}r_{j}(t)\right)
dt\hspace*{9.8em}%
\]%
\[
\leq\int_{0}^{1}\left\Vert \sum_{j=1}^{m}(A_{n-1}(x_{j}^{1},\ldots,x_{j}%
^{n-1})r_{j}(t)\right\Vert \cdot\left\Vert \sum_{j=1}^{m}\lambda_{j}x_{j}%
^{n}r_{j}(t)\right\Vert dt\hspace*{9.5em}%
\]%
\[
\leq\left(  \int_{0}^{1}\left\Vert \sum_{j=1}^{m}(A_{n-1}(x_{j}^{1}%
,\ldots,x_{j}^{n-1})r_{j}(t)\right\Vert ^{2}dt\right)  ^{1/2}\cdot\left(
\int_{0}^{1}\left\Vert \sum_{j=1}^{m}\lambda_{j}x_{j}^{n}r_{j}(t)\right\Vert
^{2}dt\right)  ^{1/2}%
\]%
\[
\leq C\Vert A_{n-1}\Vert\cdot\prod_{i=1}^{n-1}\Vert(x_{j}^{i})_{j=1}^{m
}\Vert_{\ell_{2}^{w}(E_{i})}\cdot\Vert(x_{j}^{n})_{j=1}^{m}\Vert
_{\ell_{1}^{w}(E_{n})}.\hspace*{12em}%
\]
Passing to the limit for $m\rightarrow\infty$ we get (i). The proof of (ii)
follows easily using the characterization of $E^{\prime}$ being a $GT$-spaces
of cotype 2 (see \cite[Theorem 1]{AB1} and \cite{FR}) in terms of the equality
$Rad(E^{\prime})=\ell_{2}\hat{\otimes}_{\pi}E^{\prime}$ with equivalent norms.
Hence given $(x_{j}^{i})_{j\in\mathbb{N}}\in\ell_{2}^{w}(E_{i})$ for $1\leq
i\leq n$, $m\in\mathbb{N}$ and $(\lambda_{j})_{j=1}^{m}$ such that
$|\lambda_{j}|=1$, as above we now can write, using that $\ell_{2}%
^{w}(E)\subseteq\ell_{2}^{w}(E^{\prime\prime})=(\ell_{2}\hat{\otimes}_{\pi
}E^{\prime})^{\prime}$,
\begin{align*}
\sum_{j=1}^{m}|A(x_{j}^{1},\ldots,x_{j}^{n})|  &  =\sum_{j=1}^{m}A_{n-1}%
(x_{j}^{1},\ldots,x_{j}^{n-1})(\lambda_{j}x_{j}^{n})\\
&  \leq\left\Vert \left(  A_{n-1}(x_{j}^{1},\ldots,x_{j}^{n-1})\right)
_{j=1}^m\right\Vert _{\ell_{2}\hat{\otimes}_{\pi}E_{n}^{\prime}}\cdot\left\Vert
(\lambda_{j}x_{j}^{n})_{j=1}^m\right\Vert _{\ell_{2}^{w}(E_{n})}\\
&  \leq C\left\Vert \left(  A_{n-1}(x_{j}^{1},\ldots,x_{j}^{n-1})\right)
_{j=1}^m\right\Vert _{Rad(E_{n}^{\prime})}\cdot\left\Vert (\lambda_{j}x_{j}%
^{n})_{j=1}^m\right\Vert _{\ell_{2}^{w}(E_{n})}\\
&  \leq C\Vert A_{n-1}\Vert\cdot\prod_{i=1}^{n}\Vert(x_{j}^{i})_{j=1}^m\Vert
_{\ell_{2}^{w}(E_{n})}.
\end{align*}
This finishes the proof when passing to the limit for $m\rightarrow\infty$.
\end{proof}

Now we recall that $\Pi_{as}(\ell_{1}, E)=\mathcal{L}(\ell_{1}, E)$ if and
only if $E$ has type 2 (see \cite[Theorem 21.10]{Diestel}). Therefore we
obtain the following corollary.

\begin{corollary}
Let $E$ be a Banach space such that $E^{\prime}$ has type 2. Then
\[
\mathcal{L}(\ell_{1},E;\mathbb{K})= \Pi_{(1;2,1)}(\ell_{1},E;\mathbb{K}).
\]
In particular, $\mathcal{L}(\ell_{1},\ell_{p};\mathbb{K})= \Pi_{(1;2,1)}%
(\ell_{1},\ell_{p};\mathbb{K})$ for $1<p\le2$.
\end{corollary}

An application of Lemma \ref{was} yields the following result on the structure
of some tensor products.

\begin{corollary}
Let $E$ be a Banach space such that $E^{\prime}$ has type 2. Then
there exists $C>0$ such that
\[
\Vert(x_{j}^{1}\otimes x_{j}^{2})_{j=1}^{\infty}\Vert_{\ell_{1}^{w}(\ell_1\hat{\otimes}_{\pi}E)}\leq C\Vert(x_{j}^1)_{j=1}^{\infty}\Vert_{\ell_{2}^{w}(\ell_1)}\Vert(x_{j}^2)_{j=1}^{\infty}\Vert_{\ell_{1}^{w}(E)},
\]
for all sequences $(x_{j}^1)_{j=1}^{\infty}\in\ell_{2}^{w}(\ell_1)$ and $(x_{j}^2)_{j=1}^{\infty}\in \ell_{1}^{w}(E)$.
\end{corollary}

\medskip As we have mentioned in the introduction, the fact that any operator
$T\colon E_{1}\longrightarrow E_{2}^{\prime}$ is Cohen $(p_{2}^{\prime}%
,p_{1})$-nuclear is equivalent to the coincidence $\Pi_{(1;p_{1},p_{2})}%
(E_{1},E_{2};\mathbb{K})=\mathcal{L}(E_{1},E_{2};\mathbb{K})$. In particular,
using Grothendieck's Theorem $\Pi_{(1;2,2)}(c_{0},c_{0};\mathbb{K}%
)=\mathcal{L}(c_{0},c_{0};\mathbb{K})$ we get that $\mathcal{L}(c_{0};\ell
_{1})=CN_{(2;2)}(c_{0};\ell_{1})$. Using now Lemma \ref{was} this can be
reformulated as the existence of a constant $C>0$ such that
\[
\Vert(x_{j}^{1}\otimes x_{j}^{2})_{j=1}^{\infty}\Vert_{\ell_{1}^{w}(c_0\hat{\otimes}_{\pi}c_0)}\leq C\Vert(x_{j}^1)_{j=1}^{\infty}\Vert_{\ell_{2}^{w}(c_0)}\Vert(x_{j}^2)_{j=1}^{\infty}\Vert_{\ell_{2}^{w}(c_0)},
\]
for all sequences $(x_{j}^1)_{j=1}^{\infty}, (x_{j}^2)_{j=1}^{\infty}\in \ell_{2}^{w}(c_0)$.

As pointed out before, the coincidence $\mathcal{L}(E_{1},E_{2};\mathbb{K}%
)=\Pi_{(1;p_{1},p_{2})}(E_{1},E_{2};\mathbb{K})$ is used to lift the
coincidence result to $n\geq2$. We shall show now that this condition is very
much connected to the Littlewood-Orlicz property.

\begin{definition}\rm
Let $2\le q<\infty$. We say that a Banach space $E$ satisfies the
\textit{$q$-Littlewood-Orlicz property} if $\ell_{1}^{w}(E)\subseteq\ell
_{q}\langle E\rangle$.
\end{definition}

Note that, due to Talagrand's result, spaces with the $q$-Littlewood-Orlicz
property for $q>2$ must have cotype $q$.

\begin{theorem}
\label{lo}Let $2\le q<\infty$ and $E$ be a Banach space. The following
statements are equivalent.\newline
{\rm (i)} $E^{\prime}$ has the
$q$-Littlewood-Orlicz property.
\newline
{\rm(ii)} ${\mathcal{L}%
}(X,E;\mathbb{K})=\Pi_{(1;1,q^{\prime})}(X,E;\mathbb{K})$ for any Banach space
$X$.\newline
{\rm(iii)} 
There exists $C>0$ such that
\[
\Vert(x_{j}^{1}\otimes x_{j}^{2})_{j=1}^{\infty}\Vert_{\ell_{1}^{w}(X\hat{\otimes}_{\pi}E)}\leq C\Vert(x_{j}^1)_{j=1}^{\infty}\Vert_{\ell_{1}^{w}(X)}\Vert(x_{j}^2)_{j=1}^{\infty}\Vert_{\ell_{q'}^{w}(E)}.
\]
for all sequences $(x_{j}^1)_{j=1}^{\infty}\in\ell_{1}^{w}(X)$ and $(x_{j}^2)_{j=1}^{\infty}\in \ell_{q'}^{w}(E)$.
\end{theorem}

\begin{proof} (i) $\Longrightarrow$ (ii) Let $A \colon X\times
E\longrightarrow {\mathbb K}$ be a bounded bilinear form and let $T_A \colon X\longrightarrow E'$ be the associated linear
operator. For $(x_j)_{j=1}^\infty\in \ell_1^w(X)$ and $(y_j)_{j=1}^\infty\in \ell_{q'}^w(E)$,
\begin{eqnarray*}
\sum_{j=1}^\infty|A(x_j,y_j)|&=& \sum_{j=1}^\infty| T_A(x_j)(y_j)|\\
&=& \sup_{|\alpha_j|=1}\left|\sum_{j=1}^\infty T_A(x_j)(\alpha_jy_j)\right|\\
&\le& \Vert( T_A(x_j))_{j=1}^\infty\Vert_{\ell_q\hat\otimes_\pi E'}\cdot\Vert(y_j)_{j=1}^\infty\Vert_{\ell_{q'}^w(E)}\\
&\le & C\Vert( T_A(x_j))_{j=1}^\infty\Vert_{\ell_1^w(E')}\cdot\Vert(y_j)_{j=1}^\infty\Vert_{\ell_{q'}^w(E)}\\
&\le & C\Vert A\Vert\cdot\Vert (x_j)_{j=1}^\infty\Vert_{\ell_1^w(X)}\cdot\Vert(y_j)_{j=1}^\infty\Vert_{\ell_{q'}^w(E)}.
\end{eqnarray*}
(ii) $\Longrightarrow$ (i)  Let $(x'_j)_{j=1}^\infty\in \ell_1^w(E')$ be given. Consider the bounded bilinear
form $A \colon c_0\times E\longrightarrow {\mathbb K}$ defined by the condition $A(e_j,x)= x'_j(x)$ for $x\in E$ and $j \in \mathbb{N}$.
To show that $(x'_j)_{j=1}^\infty\in \ell_q\langle E'\rangle$ it suffices to check that there is $C>0$ such that
$$\sum_{j=1}^\infty | x'_j(x_j)|\le C \Vert
(x_j)_j\Vert_{\ell_{q'}^w(E)}$$
for every $(x_j)_{j=1}^\infty \in \ell_{q'}^w(E)$. Using $X=c_0$ in the assumption,
this follows from
$$\sum_{j=1}^\infty | x'_j(x_j)|= \sum_{j=1}^\infty |A( e_j,x_j)|\le \|A\| \cdot \Vert(e_j)_{j=1}^\infty\|_{\ell_1^w(c_0)}\cdot \Vert
(x_j)_{j=1}^\infty\Vert_{\ell_{q'}^w(E)}.$$
(ii) $\Longleftrightarrow$ (iii) This is a particular case of Lemma \ref{was}.
\end{proof}

\begin{proposition}
\label{lo1}Let $2\le q<\infty$ and $E$ be a Banach space. The following
statements are equivalent.\newline
{\rm(i)} $E$ has the $q$%
-Littlewood-Orlicz property.\newline
{\rm(ii)} ${\mathcal{L}}%
(X,E;\mathbb{K})=\Pi_{(1;q^{\prime},1)}(X,E;\mathbb{K})$ for any Banach space
$X$.\newline
{\rm(iii)} 
 There exists $C>0$ such that
\[
\Vert(x_{j}^{1}\otimes x_{j}^{2})_{j=1}^{\infty}\Vert_{\ell_{1}^{w}(X\hat{\otimes}_{\pi}E)}\leq C\Vert(x_{j}^1)_{j=1}^{\infty}\Vert_{\ell_{q'}^{w}(X)}\Vert(x_{j}^2)_{j=1}^{\infty}\Vert_{\ell_{1}^{w}(E)}.
\]
for all sequences $(x_{j}^1)_{j=1}^{\infty}\in\ell_{q'}^{w}(X)$ and $(x_{j}^2)_{j=1}^{\infty}\in \ell_{1}^{w}(E)$.
\end{proposition}

\begin{proof}
(i) $\Longrightarrow$ (ii) Let $A\in {\mathcal{L}}(X,E;\mathbb{K})$ and let $T_A\colon X\longrightarrow E'$ be its associated linear map.  Let $(x_j)_{j=1}^\infty\in \ell_{q'}^w(X)$ and $(y_j)_{j=1}^\infty\in \ell_{1}^w(E)$. From (i) one has that
$(y_j)_{j=1}^\infty\in \ell_{q}\langle E\rangle$, hence
\begin{align*}
\sum_{j=1}^\infty|A(x_j,y_j)|&= \sum_{j=1}^\infty| \langle T_A(x_j),y_j\rangle |= \sup_{|\alpha_j|=1}\left|\sum_{j=1}^\infty \langle T_A(\alpha_j x_j), y_j\rangle \right|\\
&\le \Vert( T_A(x_j))_{j=1}^\infty\Vert_{\ell_{q'}^w(E')}\cdot \Vert(y_j)_{j=1}^\infty\Vert_{\ell_q\langle E\rangle }\\
&\le  \|A\|\cdot \Vert( x_j)_{j=1}^\infty\Vert_{\ell_{q'}^w(X)}\cdot \Vert(y_j)_{j=1}^\infty\Vert_{\ell_{1}^w(E)}.
\end{align*}
(ii) $\Longrightarrow$ (i)  Let $(x_j)_{j=1}^\infty\in \ell_1^w(E)$ be given. To show that $(x_j)_{j=1}^\infty\in \ell_q\langle E\rangle$, fix $(x'_j)_{j=1}^\infty\in \ell_{q'}^w(E')$  and consider the bounded bilinear
form $A \colon \ell_q\times E\longrightarrow {\mathbb K}$ defined by the condition $$A((\lambda_j),x)= \sum_j\lambda_jx'_j(x)$$ for $x\in E$ and $(\lambda_j) \in \ell_{q}$. Clearly
$\|A\|= \|(x'_j)_j\|_{\ell_{q'}^w(E')}$ and $A(e_j,x_j)=x'_j(x_j)$.
From assumption we know that $A\in \Pi_{(1;q',1)}(\ell_q,E;\mathbb{K})$, from which it follows that
$$\sum_{j=1}^\infty | x'_j(x_j)|\le \|A\|\cdot \Vert
(e_j)_j\Vert_{\ell_{q'}^w(\ell_q)}\cdot\Vert
(x_j)_j\Vert_{\ell_{1}^w(E)}=\|(x'_j)_j\|_{\ell_{q'}^w(E')}\cdot\Vert
(x_j)_j\Vert_{\ell_{1}^w(E)}$$
for every $(x_j)_{j=1}^\infty \in \ell_{1}^w(E)$. This shows the result.
\noindent The equivalence (ii) $\Longleftrightarrow$ (iii)  has been shown in Lemma \ref{was}.
\end{proof}

\bigskip

\noindent\textbf{Acknowledgement.} The authors thank the referee for her/his important suggestions that improved substantially the final presentation of the paper.

\vspace{2mm}

\noindent[Oscar Blasco] Departamento de An\'alisis Matem\'atico, Universidad
de Valencia, 46.100 Burjasot - Valencia, Spain, e-mail: oscar.blasco@uv.es

\medskip

\noindent[Geraldo Botelho] Faculdade de Matem\'atica, Universidade Federal de
Uberl\^andia, 38.400-902 - Uberl\^andia, Brazil, e-mail: botelho@ufu.br

\medskip

\noindent[Daniel Pellegrino] Departamento de Matem\'atica, Universidade
Federal da Para\'iba, 58.051-900 - Jo\~ao Pessoa, Brazil, e-mail:
dmpellegrino@gmail.com and pellegrino@pq.cnpq.br

\medskip

\noindent[Pilar Rueda] Departamento de An\'alisis Matem\'atico, Universidad de
Valencia, 46.100 Burjasot - Valencia, Spain, e-mail: pilar.rueda@uv.es
\end{document}